\documentclass[11pt,reqno]{amsart}
\usepackage{amssymb,latexsym}
\usepackage{graphicx}
\usepackage{mathtools}
\usepackage{color}
\usepackage[T1]{fontenc}
\usepackage{hyperref}
\usepackage{amsmath}
\usepackage{breqn}
\usepackage[toc,page]{appendix}
\usepackage{url}    
\usepackage{cancel}
\usepackage{multirow}
\newcommand{\bburl}[1]{\textcolor{blue}{\url{#1}}}

\calclayout

\makeatletter
\newcommand{\monthyear}[1]{%
  \def\@monthyear{\uppercase{#1}}}
\newcommand{\volnumber}[1]{%
  \def\@volnumber{\uppercase{#1}}}
\AtBeginDocument{%
\def\ps@plain{\ps@empty
  \def\@oddfoot{\@monthyear \hfil \thepage}%
  \def\@evenfoot{\thepage \hfil \@volnumber}}
\def\ps@firstpage{\ps@plain}
\def\ps@headings{\ps@empty
  \def\@evenhead{%
    \setTrue{runhead}%
    \def\thanks{\protect\thanks@warning}%
    \uppercase{\ }\hfil}%
  \def\@oddhead{%
    \setTrue{runhead}%
    \def\thanks{\protect\thanks@warning}%
    \hfill\uppercase{Bounds on Zeckendorf games}}%
  \def\@evenhead{%
    \setTrue{runhead}%
    \def\thanks{\protect\thanks@warning}%
    \uppercase{The Fibonacci Quarterly}\hfil}%
  \let\@mkboth\markboth
  \def\@evenfoot{%
    \thepage \hfil \@volnumber}%
  \def\@oddfoot{%
    \@monthyear \hfil \thepage}%
  }%
\footskip=25pt
\pagestyle{headings}%
}
\makeatother

\theoremstyle{plain}
\numberwithin{equation}{section}
\newtheorem{thm}{Theorem}[section]

\newcommand{\ignore}[1]{}

\newcommand\scr{\scriptstyle}
\newcommand\ds{\displaystyle}

\newcommand\be{\begin{eqnarray}}
\newcommand\ee{\end{eqnarray}}
\newcommand\bea{\begin{eqnarray}}
\newcommand\eea{\end{eqnarray}}
\newcommand\ben{\begin{enumerate}}
\newcommand\een{\end{enumerate}}

\newtheorem{lem}[thm]{Lemma}

\newtheorem{defi}[thm]{Definition}

\newtheorem{rek}[thm]{Remark}

\begin{document}

\monthyear{TBD}
\volnumber{Volume, Number}
\setcounter{page}{1}

\title{Bounds on Zeckendorf Games}
\author{ANNA CUSENZA, AIDEN DUNKELBERG, KATE HUFFMAN, DIANHUI KE,\\ MICAH MCCLATCHEY, STEVEN J. MILLER, CLAYTON MIZGERD, VASHISTH TIWARI,\\JINGKAI YE, AND XIAOYAN ZHENG}



\address{University of California, Los Angeles, Los Angeles, CA 90095}
\email{ascusenza@g.ucla.edu}

\address{Williams College, Williamstown, MA 01267}
\email{awd4@williams.edu}





\address{University of Alabama, Tuscaloosa, AL 35401}
\email{klhuffman@crimson.ua.edu}


\address{University of Michigan, Ann Arbor, MI 48109}
\email{kdianhui@umich.edu}





\address{Houghton College, Houghton, NY 14744}
\email{mcclatch87@gmail.com}

\address{Department of Mathematics and Statistics, Williams College, Williamstown, MA 01267}
\email{sjm1@williams.edu}

\address{Department of Mathematics and Statistics, Williams College, Williamstown, MA 01267}
\email{cmm12@williams.edu}



\address{University of Rochester, Rochester, NY 14627}
\email{vtiwari2@u.rochester.edu}



\address{Whitman College, 280 Boyer Avenue, Walla Walla, WA, 99362}
\email{yej@whitman.edu.edu}


\address{Washington University in St. Louis, St. Louis, MO 63130}
\email{zhengxiaoyan@wustl.edu}


\date{\today}
\begin{abstract}
Zeckendorf proved that every positive integer $n$ can be written uniquely as the sum of non-adjacent Fibonacci numbers. We use this decomposition to construct a two-player game. Given a fixed integer $n$ and an initial decomposition of $n=n F_1$, the two players alternate by using moves related to the recurrence relation $F_{n+1}=F_n+F_{n-1}$, and whoever moves last wins. The game always terminates in the Zeckendorf decomposition; depending on the choice of moves the length of the game and the winner can vary, though for $n\ge 2$ there is a non-constructive proof that Player 2 has a winning strategy.

Initially the lower bound of the length of a game was order $n$ (and known to be sharp) while the upper bound was of size $n \log n$. Recent work decreased the upper bound to of size $n$, but with a larger constant than was conjectured. We improve the upper bound and obtain the sharp bound of $\frac{\sqrt{5}+3}{2}\ n - IZ(n) - \frac{1+\sqrt{5}}{2}Z(n)$, which is of order $n$ as $Z(n)$ is the number of terms in the Zeckendorf decomposition of $n$ and $IZ(n)$ is the sum of indices in the Zeckendorf decomposition of $n$ (which are at most of sizes $\log n$ and $\log^2 n$ respectively). We also introduce a greedy algorithm that realizes the upper bound, and show that the longest game on any $n$ is achieved by applying splitting moves whenever possible.
\end{abstract}
\maketitle

\tableofcontents


\section{Introduction}
The Fibonacci numbers are one the most interesting and famous sequences.
Among their fascinating properties, the Fibonacci numbers lend themselves to a beautiful theorem by Edouard Zeckendorf \cite{Ze} which states that each positive integer $n$ can be written uniquely as the sum of distinct, non-consecutive Fibonacci numbers. This sum is called the \textit{Zeckendorf decomposition} of $n$ and requires that we define the Fibonacci numbers by $F_1 = 1, F_2 = 2, F_3 = 3, F_4 = 5...$ instead of the usual $1, 1, 2, 3, 5...$ to create uniqueness.
Baird-Smith, Epstein, Flint and Miller \cite{BEFMgen,BEFM} create a game based on the Zeckendorf decomposition. We quote from \cite{BEFM} to describe the game.

We introduce some notation. By $\{1^n\}$ or $\{{F_1}^n\}$ we mean $n$ copies of $1$, the first Fibonacci number. If we have 3 copies of $F_1$, 2 copies of $F_2$, and 7 copies of $F_4$, we write either $\{{F_1}^3 \wedge {F_2}^2 \wedge {F_4}^7 \}$ or $\{1^3 \wedge 2^2 \wedge 5^7\}$.
\begin{defi}[The Two Player Zeckendorf Game]\label{defi:zg}
At the beginning of the game, there is an unordered list of $n$ 1's. Let $F_1 = 1, F_2 = 2$, and $F_{i+1} = F_i + F_{i-1}$; therefore the initial list is $\{{F_1}^n\}$. On each turn, a player can do one of the following moves.
\begin{enumerate}
\item If the list contains two consecutive Fibonacci numbers, $F_{i-1}, F_i$, then a player can change these to $F_{i+1}$. We denote this move $\{F_{i-1} \wedge F_i \rightarrow F_{i+1}\}$.
\item If the list has two of the same Fibonacci number, $F_i, F_i$, then
\begin{enumerate}
\item if $i=1$, a player can change $F_1, F_1$ to $F_2$, denoted by $\{F_1 \wedge F_1 \rightarrow F_2\}$,
\item if $i=2$, a player can change $F_2, F_2$ to $F_1, F_3$, denoted by $\{F_2 \wedge F_2 \rightarrow F_1 \wedge F_3\}$, and
\item if $i \geq 3$, a player can change $F_i, F_i$ to $F_{i-2}, F_{i+1}$, denoted by $\{F_i \wedge F_i \rightarrow F_{i-2}\wedge F_{i+1} \}$.
\end{enumerate}
\end{enumerate}
The players alternate moving. The game ends when one player moves to create the Zeckendorf decomposition.
\end{defi}
The moves of the game are derived from the Fibonacci recurrence, either combining terms to make the next in the sequence or splitting terms with multiple copies. A proof that this game is well defined and ends at the Zeckendorf decomposition can be found in \cite{BEFM}.

\ \\
We introduce some further notation and state some simple results.

\begin{itemize}

\item Let $i_{\max}(n)$ be the largest index of terms in the Zeckendorf decomposition of $n$. The order of $i_{\max}(n)$ is at most $\log n$; this follows immediately from the exponential growth of the Fibonacci numbers, as we can never use a summand larger than the original number $n$.\\ \

\item Let $\delta_i$ denote the number of $F_i$'s in the Zeckendorf decomposition of $n$. Then
$n=\sum_{i=1}^{i_{\max}(n)}\delta_i F_i$.\\ \

\item Let $Z(n)$ denote the number of terms in the Zeckendorf decomposition of $n$, and $Z(n)=\sum_{i=1}^{i_{\max}(n)}\delta_i $. The order of
$Z(n)$ is at most $\log n$ since $Z(n)\leq i_{\max}(n)$.\\ \

\item Let $IZ(n)$ denote the sum of indices in the Zeckendorf decomposition of $n$, and $IZ(n)=\sum_{i=1}^{i_{\max}(n)} i\, \delta_i $.
The order of $IZ(n)$ is at most $\log^2 n$; this follows trivially from summing the indices and recalling the largest index used is of order $\log n$.\\ \

\item The original upper bound for the game was of order $n\log n$, and the lower bound was found to be sharp at $n-Z(n)$ in  \cite{BEFM}. The upper bound on the number of moves was improved to $3n-3Z(n)-IZ(n)+1$ in \cite{LLMMSXZ}. Since the order of $Z(n)$ and $IZ(n)$ are both less than $n$, we observe that the upper and lower bounds are both of order $n$.\\ \

\item Finally, several deterministic games have been introduced in \cite{LLMMSXZ}. These are defined in terms of the priority of moves; that is, each move in a strategy will follow whichever move is available and comes first in the ordering of moves.
\begin{itemize}
 \item Combine Largest: adding consecutive indices from largest to smallest, adding 1's, splitting from largest to smallest.
 \item Split Largest: splitting from largest to smallest, adding consecutive indices from largest to smallest, adding 1's.
 \item Split Smallest: splitting from smallest to largest, adding 1's, adding consecutive indices from smallest to largest.
\end{itemize}
It was shown in the same paper that the Combine Largest and Split Largest games both realize the shortest game.
\end{itemize}

\ \\

Since the lower bound of the game has been shown to be sharp, we focus on the upper bound of the game.
One of our main result is a proof of a conjecture from \cite{BEFM} that the longest game on any $n$ is achieved by applying splitting moves whenever possible.
\begin{thm}\label{thm:strategy}
The longest game on any $n$ is achieved by applying split moves or combine 1's (in any order) whenever possible, and, if there is no split or combine 1 move available, combine consecutive indices from smallest to largest.
\end{thm}

This algorithm is not deterministic. Thus, there are many game paths that follow this algorithm.
For instance, it can be easily shown that the Split Smallest game described in \cite{LLMMSXZ} is a deterministic example of this algorithm, and therefore realizes the longest game.

Now that we have an algorithm that achieves the longest game, we are interested in the upper bound of the game length.
The previous upper bound was already very close to the known lower bound (both of order $n$). Nevertheless, we are able to further close the gap.

\begin{thm}\label{thm:bound}
Let $a_i = F_{i+2} - i - 2\: (i>0,\:i\in\mathbb{N})$. The upper bound of the game is given by $\sum_{i=1}^{i_{\max}(n)}a_i\ \delta_i$
which is at most $\frac{\sqrt{5}+3}{2}\ n\ -\ IZ(n)\ -\ \frac{1+\sqrt{5}}{2}Z(n)$.
\end{thm}

It was originally conjectured in \cite{BEFM} that the number of moves in the Split Smallest game grows linearly with $n$, with a constant of the golden mean squared, which is equivalent to $\frac{\sqrt{5}+3}{2}$. We observe that this conjecture has been shown here, since the order of $Z(n)$ and $IZ(n)$ are less than $n$.

Though this bound is very close to the actual longest game, it is not a strict upper bound for most $n$.
In fact, we observed during the construction of this upper bound that this bound is sharp if and only if the game on $n$ can be played with only split and combine 1 moves, and identified all such $n$.
\begin{thm}\label{thm:splitGame}
A game can be played with only splitting and combine 1 moves, if and only if $n\ =\ F_k-1\ (k\geq2)$.
\end{thm}
In Section \ref{sec:strategy} we prove Theorem \ref{thm:strategy}. Then in Section \ref{sec:upper bound} we prove Theorems \ref{thm:bound} and \ref{thm:splitGame}. Finally in Section \ref{sec:conclusion} we give some possible directions for future research.

\section{Strategy to achieve the longest game}\label{sec:strategy}
We start by introducing some notation that we use in our proofs.

Following the notation introduced in \cite{LLMMSXZ},
we let $MC_i$ denote the number of combining moves at the index $i$ with $i\geq2$, with $MC_1$ the number of combine 1 moves. Similarly the number of splitting moves at $i$ is denoted $MS_i$ for $i\geq2$. We refer to combining moves at $i$ by $C_i$, and splitting moves at $i$ by $S_i$.

The move of adding 1's is usually considered a combining move as the case in \cite{BEFMgen,BEFM,LLMMSXZ}, but for the sake of our proofs, we consider combine 1 ($C_1$) to be a splitting move, and also refer to it as $S_1$ in this section.

We begin with the proof of Theorem \ref{thm:strategy}, which is a greedy algorithm that achieves the longest game path, starting with two lemmas.

As a reminder, the algorithm requires moves to be in the following order: choose any split or combine $1$ move whenever possible, then combine consecutive terms with smallest indices.
\begin{lem}\label{lem:findSameLen}
If the aforementioned strategy gives us choice of moves at some game state $G$, then starting from this game state, no matter which move we choose at this step, there exists a path that follows our strategy and has the same length as our initial path.
\end{lem}
\begin{proof}
In paths following our strategy, the only game states that allow choice of moves are game states with at least two splitting (including $S_1$) moves available.

Let $P$ be a game path starting from any game state that follows our strategy, and let $G$ be any game state that $P$ visits that allows a choice.
Let $P'$ be another path that starts from the same game state as $P$ and follows our strategy such that $P'$ follows the same moves as $P$ until they differ in choice of move for the first time at $G$.

Let $P$ choose $S_i$ and $P'$ choose $S_j$ at $G$. We want to construct the steps for $P'$ such that it has same length as $P$.

Since $P'$ could have chosen $S_i$, we must have at least two $F_i$ at $G$. Since $S_j$ does not decrease the number of $F_i$, we still have at least two $F_i$ after this step. Thus for the next step, $P'$ can still choose $S_i$.
By similar reasons, $P'$ can always imitate the moves that $P$ takes after $S_i$ until $P$ takes a $S_j$.
We know that $P$ must take $S_j$ at some step because we had at least two $F_j$ at $G$, and the only moves that decrease the number of $F_j$'s are $S_j$, $C_j$, and $C_{j+1}$. Since our strategy prioritizes split moves over combine moves, we must take $S_j$ at some point in $P$.

Thus the moves in $P'$ after $G$ are as follows: perform $S_j$ and $S_i$ in the first two steps, and then imitate the moves of $P$ after $S_i$ until $P$ takes $S_j$. Since $P$ follows our strategy of prioritizing split  moves, $P'$ does also.
In this way, $P$ and $P'$ take the same set of steps but in different order, so they reach the same game state with the same number of steps. After that, $P$ and $P'$ follow exactly same steps until game terminates.

Thus, we prove that no matter which move we choose at some game state $G$, there exists a path that follows our strategy and has same length as our initial path.
\end{proof}

\begin{lem}\label{lem:split order}
Starting from any game state, all paths that follow our strategy have the same length.
\end{lem}
\begin{proof}
We show that an arbitrary game path $P$ that follows our strategy is no longer or shorter than any other path that follow our strategy, given that they start from same game state.

Suppose for the sake of contradiction that $P$ and $P'$ both follow our strategy but differ in length. Then $P$ and $P'$ must differ by at least $1$ move. Let $G$ be the first game state where $P$ and $P'$ differ in choice of moves. By Lemma \ref{lem:findSameLen}, there exists a path $P_1$ that chooses the same move as $P'$ at $G$, has the same length as $P$, and follows our strategy.

Since $P_1$ and $P'$ differ in length, they must differ by at least one move. Thus there exists a game state $G_1$ after $G$ where $P_1$ and $P'$ differ in their choice for the first time. Again, by Lemma \ref{lem:findSameLen}, there exists a path $P_2$ that chooses the same move as $P'$ at $G_1$, has the same length as $P_1$ (which is equal to length of $P$), and follows our strategy.

Since the number of steps in any game path is finite, we can repeat this process until we find a $P_k$ with the same length as $P$, but there no longer exists $G_k$ where $P'$ and $P_k$ can differ in choice of moves. Thus, the rest of steps are deterministic, which means $P_k$ and $P'$ must be exactly the same path, and therefore both have the same length as $P$, a contradiction.

In conclusion, all paths that start from same game state and follow our strategy have same length.
\end{proof}

\begin{lem}\label{lem:otherPath}
Starting from any game state, if a path does not follow our strategy, then this game path is either not the longest path or there exists a path that has the same length as this path and follows our strategy.
\end{lem}
\begin{proof}
Suppose a game path $P$ contains at least one step that is not chosen by our strategy.
We want to either find a path $P^*$ that is longer than $P$, or construct a $P^*$ that follows our strategy and is as long as $P$.

We look at the last step in $P$ that is not chosen by our strategy. Suppose the step is taken at game state $G$.
Since we consider the last step that does not follow our strategy, all moves after this step must follow our strategy.

There are two situations when a step is not chosen by our strategy: either the combining move taken is not the one with smallest index when no splitting move is available,
or a combining move is taken when there is splitting move available. Note that in splitting move we include combine $1$'s moves, and in combining move we exclude combine $1$'s.
We look at these two cases. In both cases we want to find a path $P'$ that is either longer than $P$ or follows our strategy at and after $G$ and have same length as $P$.

First, suppose $P$ takes a combining move that is not the smallest when no splitting move is available at $G$. In this case, there is at most one $F_k$ for any $k$.
Let $C_i$ ($i\geq2$) be the smallest combining move at $G$ and let $C_j$ ($j>i$) be the combining move chosen by $P$ at $G$. We study the following sub-cases based on whether $G$ contains a $F_{i-2}$ term.
\begin{itemize}
\setlength{\itemsep}{5pt}
    \item \textit{Case 1.1.} At $G$, there is at least one $F_{j-2}$.

    In this case, we find a path $P'$ that is longer than $P$. Let $P'$ take the same moves as $P$ before reaching $G$. At $G$, path $P$ takes $C_j$ at $G$ and reaches $G'$.
    Compared to $G$, $G'$ contains one less $F_{j-1}$, one less $F_j$, and one more $F_{j+1}$.
    Let path $P'$ take $C_{j-1}$ and $S_{j}$.
    Then the game state it reaches also contains one less $F_{j-1}$, one less $F_j$, and one more $F_{j+1}$ compared to $G$. Thus $P'$ also reaches $G'$ with one more step than $P$. After that, $P'$ can imitate the moves $P$ takes until game terminates. In the end, $P'$ is one move longer than $P$.
    \item \textit{Case 1.2.} At $G$, there are no $F_{j-2}$.

    Again, let $P'$ follow same steps as $P$ and reach the game state $G$. At $G$, let $P$ take $C_j$ and $P'$ take $C_i$. After that, $P$ follows our strategy.

    First, we look at the steps in $P$. $C_j$ increases number of $F_{j+1}$ by one, so if there is one $F_{j+1}$ at $G$, we apply $S_{j+1}$. This is the only possible splitting move at this game state since we assumed there are no $F_{j-2}$.
    Similarly, $S_{j+1}$ increases $F_{j+2}$ by one, so if there is one $F_{j+2}$ at $G$, we apply $S_{j+2}$.
    Since we have used up all $F_{j-1}$'s with $C_j$, there are no $F_{j-1}$'s before taking $S_{j+1}$, so we cannot apply $S_{j-1}$. Thus $S_{j+2}$ should be the only possible splitting move at this game state.

    We repeat the process of taking the only splitting move until we have to do a combining move.
    Note that the number of such splitting moves that can be taken by $P$ after taking $C_j$ and before taking another combining move depends on the number of consecutive $F_k$'s ($k>j$) starting from $F_{j+1}$. Let $\alpha\geq 0$ be the number of such moves. The move taken in $P$ after these $\alpha$ steps is $C_i$ since there is no splitting move available and $C_i$ is the smallest combining move.

    Now $P$ has taken $1+\alpha+1$ steps and $P'$ has taken 1 step.
    The current game state for $P$ and $P'$ has the same number of $F_k$'s for all $k\leq j-3$ since $S_j$ and the $\alpha$ splitting moves in $P$ do not affect the number of $F_k$'s ($k\leq j-3$), and $P$ and $P'$ both take a $C_i$ move.
    Thus, they can follow the same steps until either we get a $F_{j-2}$ before we have to take $C_j$ in $P'$ or we never get a $F_{j-2}$ and the next step in $P'$ is $C_j$. Suppose we took $\beta$ steps before we stop. We look at these two cases.

    \vspace{5pt}
    \begin{itemize}
    \setlength{\itemsep}{5pt}
        \item \textit{Case 1.2.1.} We get a $F_{j-2}$ before we have to take $C_j$ in $P'$.

        This case is similar to Case 1.1 where we find a path longer than $P$. In $P'$ we take $C_{j-1}$, $S_j$, and then follow the $\alpha$ steps described in $P$. After that, $P$ and $P'$ reach the same game state. Notice that after $G$, path $P$ took $1+\alpha+1+\beta$ steps, and $P'$ took $1+\beta+2+\alpha$ steps. Thus, by the end of the game, $P'$ is one move longer than $P$.

        \item \textit{Case 1.2.2.} We never get a $F_{j-2}$ before the next step in $P'$ is $C_{j}$.

        In this case $P'$ takes the $C_j$ move and then follows the $\alpha$ steps described in $P$. Then $P$ and $P'$ reach the same game state, and share the same steps after that. Notice that the sub-paths of $P$ and $P'$ after $G$ follow our strategy, and $P'$ has the same length as $P$.
    \end{itemize}
\end{itemize}

Second, suppose $P$ chooses a combining move when there is a splitting move available. Let $C_i$ ($i\geq 2$) be the combining move that $P$ chooses. Then there must be at least one $F_{i-1}$ and one $F_{i}$ at $G$. We consider the following cases based on the number of $F_{i-1}$'s and $F_{i}$'s.
\begin{itemize}
    \setlength{\itemsep}{4pt}
    \item \textit{Case 2.1.} There are more than one $F_{i-1}$'s at $G$ (i.e., $P'$ can take the move $S_{i-1}$).
    \begin{itemize}
        \setlength{\itemsep}{5pt}

        \item \textit{Case 2.1.1.} $i=2$ (i.e., path $P$ takes $C_2$).

        At $G$, path $P$ takes $C_2$, and path $P'$ takes $S_1$ and $S_2$ to reach the same game state. After that $P'$ imitates the steps $P$ takes. In the end, $P'$ is one move longer than $P$.
        \item \textit{Case 2.1.2.} $i=3$ (i.e., path $P$ takes $C_3$).

        At $G$, path $P$ takes $C_3$, and path $P'$ takes $S_2$, $S_3$, $S_1$ to reach the same game state. After that $P'$ imitates the steps $P$ takes. In the end, $P'$ is two moves longer than $P$.
        \item \textit{Case 2.1.3.} $i>3$.

        At $G$, path $P$ takes $C_i$, and path $P'$ takes $S_{i-1}$, $S_i$, $C_{i-2}$ to reach the same game state. After that $P'$ imitates the steps $P$ takes. In the end, $P'$ is two moves longer than $P$.
    \end{itemize}
    \item \textit{Case 2.2.} There is exactly one $F_{i-1}$ and more than one $F_i$ at $G$ (i.e., $P'$ can take $S_i$).
    \begin{itemize}
        \setlength{\itemsep}{5pt}

        \item  \textit{Case 2.2.1.} $i=2$ (i.e., path $P$ takes $C_2$).

        At $G$, path $P$ takes $C_2$, and path $P'$ takes $S_2$ and $S_1$ to reach the same game state. After that $P'$ imitates the steps $P$ takes. In the end, $P'$ is one move longer than $P$.
        \item  \textit{Case 2.2.2.} $i>2$.

        At $G$, path $P$ takes $C_i$, and path $P'$ takes $S_i$ and $C_{i-1}$ to reach the same game state. After that $P'$ imitates the steps $P$ takes. In the end, $P'$ is one move longer than $P$.
    \end{itemize}
    \item \textit{Case 2.3.} There is exactly one $F_{i-1}$ and one $F_i$ at $G$.

    Let $G'$ be the game state $P$ reaches after taking $C_i$ at $G$. Since $G$ is the last game state where our strategy is violated, all steps in $P$ after $G$ follows our strategy.
    We proved in Lemma \ref{lem:split order} that starting from any game state, any game path that follows our strategy has same length. Thus, there exists a path that starts from $G'$, performs $S_{i+1}$ only when no other splitting move is available, and has the same length as the sub-path of $P$ starting from $G'$.
    We extend this path so that it starts from the same game state as $P$ and follows the same steps as $P$ until $G'$.
    Call the extended path $P''$.

    Note that the only differences between $G$ and $G'$ are that $G'$ has one more $F_{i+1}$ term, one less $F_i$ term and one less $F_{i-1}$ term.
    If we let $P'$ follow the same moves after $G$ that $P$ does after $G'$, then the game states in $P'$ and $P$ always differ in only these three terms.
    Thus, $P'$ may be unable to imitate $P$ if $P$ performs $S_{i+1}$.

    We avoid this problem to the maximum extent by letting $P'$ follow the same steps in $P''$ (instead of $P$) until either $P''$ is forced to take $S_{i+1}$ and $P'$ cannot choose the same step, or $P'$ performs all splitting moves before $P''$ takes a combining move.
    In both situations,
    let $G''$ be the game state reached by $P''$ here and consider the next step in $P'$. The next step is either a splitting move or a combining move (there must exist a next step since we can always take $C_i$). Notice that if there is a splitting move possible, it is either $S_{i-1}$ or $S_i$ since $F_{i-1}$ and $F_i$ are the only two terms which game states in $P'$ have more of than game states in $P''$.
    \vspace{4pt}
    \begin{itemize}
    \setlength{\itemsep}{5pt}
        \item \textit{Case 2.3.1.} The next step in $P'$ can be $S_{i-1}$ or $S_i$.

        Suppose the next step taken in $P'$ is $C_i$.
        If the step in $P''$ before reaching $G''$ is $S_{i+1}$ and $P'$ failed to follow this move, $P'$ takes $S_{i+1}$ after the $C_i$ move.
        Then $P'$ reaches $G''$ with same number of steps as $P''$.
        However, according to \textit{Case 2.1} and \textit{Case 2.2}, we can find a path that is longer than $P'$. Thus there exist a path longer than $P$.

        \item \textit{Case 2.3.2.} The next step in $P$ can only be a combining move.

        Again, suppose the next step taken in $P'$ is $C_i$ (and $S_{i+1}$ if $P'$ failed to follow the $S_{i+1}$ move in $P''$). Then, it reaches $G''$ with same number of steps as $P''$. Let $P'$ follow the same moves in $P''$ after $G''$.

        If $C_i$ was not the smallest possible combining move $P'$ could have taken, then by our discussion about the case where the smallest combining move is not chosen, we know that there exists either a path longer than $P'$ or a path that has same length as $P'$ and its sub-path (after $G$) follows our strategy.

        If $C_i$ was the smallest possible combining move, then this step followed our strategy.
        Thus $P'$ is a path that has the same length as $P$ and whose sub-path (after $G$) follows our strategy.
    \end{itemize}
    In both cases, we can either find a path that is longer than $P$ or a path whose sub-path after $G$ follows our strategy and have same length as $P$.
\end{itemize}

Now we can start to construct a path $P^*$ that is either longer than $P$ or has same length as $P$ and follows our strategy.

Since the game takes finitely many steps, there are finitely many game states in $P$ that do not follow our strategy in choosing moves. We denote all such game states in $P$ in reverse chronological order as $G_1,G_2,\dots,G_k$, where $G_1$ is the last game state where a move is not chosen in accordance with our strategy and $G_k$ is the first.

We start by looking at the path after $G_1$. Since our strategy is not followed at $G_1$, it must fall into either of the two cases discussed above. In both cases, we can either find a path that is longer than $P$ or find a path whose sub-path after $G_1$ always follow our strategy and is as long as $P$.

If we find a path that is longer than $P$, then this is the $P^*$ that we are looking for.
Otherwise, we denote the path whose sub-path after $G_1$ follows our strategy as $P_1$.
Notice that $P$ and $P_1$ follow same moves before reaching $G_1$, and $P_1$ follows our strategy at  $G_1$, so the last game state in $P_1$ where our strategy is violated is $G_2$.

By the same argument, we can either find a path that is longer than $P_1$ or find a path whose sub-path after $G_2$ always follow our strategy and is as long as $P_1$.
If we find a path that is longer than $P_1$, then this is the $P^*$ that we are looking for. Otherwise, we find a $P_2$ whose sub-path after $G_2$ follows our strategy and has same length as $P_1$.

We repeat this process until either we find a $P^*$ that is longer than $P$, or we find a $P_k$ whose sub-path after $G_k$ follows our strategy and has the same length as $P$.
Since $G_k$ is the first game state in $P$ where our strategy is not followed, $P_k$ is a path that follows our strategy from the starting state. Thus $P_k$ is the $P^*$ we want.

In conclusion, starting from any game state, if there is a path $P$ that does not follow our strategy, then it is either not the longest game or there exists a path $P'$ that is as long as $P$ and follows our strategy.
\end{proof}

\begin{proof}[Proof of Theorem \ref{thm:strategy}]
By Lemma \ref{lem:otherPath}, we proved that a path that does not follow our strategy is either not the longest game, or there exists a path that follows our strategy and is as long as the original path. By Lemma \ref{lem:split order}, we know that all paths that start from same game state and follow our strategy have the same length. Thus our strategy gives the longest game.
\end{proof}


\section{Upper Bound on the Game Length}\label{sec:upper bound}
In this section, we construct and analyze the order of an upper bound on the game length.
\begin{proof}[Proof of Theorem \ref{thm:bound}]
The first step is to construct an upper bound on the game length.
To do this, we first look at changes in the amount of $F_2$. We start the game with no $F_2$, and end with $\delta_2$ of $F_2$. Every time we combine two $F_1$'s, we get a $2$ ($F_1\wedge F_1\rightarrow F_2$) and increase the number of $F_2$'s by one. Every time we split two $F_4$'s, we get a $2$ ($F_4 \wedge F_4\rightarrow F_2\wedge F_5$) and increase the number of $F_2$'s by one. These are the only moves that increase the number of $F_2$'s.
Each combining move of $F_1$ and $F_2$ ($F_1\wedge F_2\rightarrow F_3$) and combining move of $F_2$ and $F_3$ ($F_2\wedge F_3\rightarrow F_4$) decreases the number of $F_2$'s by one. Finally splitting two $F_2$'s ($F_2\wedge F_2\rightarrow F_1\wedge F_3$) decreases the number of $F_2$'s by two. These are the only moves that decrease the number of $F_2$'s.

Thus we can construct the following equation:
\begin{equation}
MC_1-2MS_2+MS_4-MC_2-MC_3=\delta_2.
\end{equation}

Similarly, for each $3\leq i\leq i_{\max}(n)$, we start the game with no $F_i$, and end with $\delta_i$ of the $F_i$. Every time we combine $F_{i-2}$ and $F_{i-1}$, we increase the number of $F_i$'s by one. Every time we split two $F_{i-1}$'s, we increase the number of $F_i$'s by one. Every time we split two $F_{i+2}$'s, we increase the number of $F_i$'s by one. These are the only moves that increase the number of $F_i$'s.
Each combining move of $F_{i-1}$ and $F_{i}$ and combining move of $F_{i}$ and $F_{i+1}$  decreases the number of $F_i$'s by one. Finally splitting two $F_{i}$'s decreases the number of $F_i$'s by two. These are the only moves that decrease the number of $F_i$'s.

Thus we have for $3\leq i\leq i_{\max}(n)$
\begin{equation}
MS_{i-1}-2MS_{i}+MS_{i+2}+MC_{i-1}-MC_{i}-MC_{i+1}\ =\ \delta_i .
\end{equation}

Since $i_{\max}(n)$ is the largest index in the final decomposition, we know that for all $i\geq i_{\max}(n)$, $MS_i=MC_i=0$.
Thus for $i\geq i_{\max}(n)-2$, we can get rid of a few terms in the equation above:
\begin{equation}
MS_{i_{\max}(n)-3}-2MS_{i_{\max}(n)-2}+MC_{i_{\max}(n)-3}-MC_{i_{\max}(n)-2}-MC_{i_{\max}(n)-1}=\delta_{i_{\max}(n)-2},
\end{equation}
\begin{equation}
MS_{i_{\max}(n)-2}-2MS_{i_{\max}(n)-1}+MC_{i_{\max}(n)-2}-MC_{i_{\max}(n)-1}\ =\  \delta_{i_{\max}(n)-1},
\end{equation}
\begin{equation}
MS_{i_{\max}(n)-1}+MC_{i_{\max}(n)-1}\ =\ \delta_{i_{\max}(n)}.
\end{equation}

Now we have $i_{\max}(n)-1$ linear equations with $2\cdot i_{\max}(n)-3$ variables, we write the system of equations in matrix form.
\begin{equation}\label{eq:syseq}
\left[
\begin{tabular}{c c c c c c c c c c c c c}
1 & -2 & 0 & 1 & $\cdots$ & 0 & 0 & -1 & -1 & 0 & $\cdots$ & 0 & 0 \\
0 & 1 & -2 & 0 & $\cdots$ & 0 & 0 & 1 & -1 & -1 & $\cdots$ & 0 & 0 \\
0 & 0 & 1 & -2 & $\cdots$ & 0 & 0 & 0 &  1 & -1 & $\cdots$ & 0 & 0 \\
$\vdots$&&&&$\cdots$ & & $\vdots$&$\vdots$&&& $\cdots$ & &$\vdots$\\
0 & 0 & 0 & 0 & $\cdots$ & 1 & -2 & 0 & 0 & 0 & $\cdots$ & 1 & -1 \\
0 & 0 & 0 & 0 & $\cdots$ & 0 & 1 & 0 & 0 & 0 & $\cdots$ & 0 & 1 \\
\end{tabular}
\right]\!\!\!
\left[
{\begin{tabular}{c}
$MC_1$ \\
$MS_2$ \\
$MS_3$ \\
$\vdots$ \\
$MS_{i_{\max}(n)-1}$ \\
$MC_2$ \\
$\vdots$ \\
$MC_{i_{\max}(n)-1}$
\end{tabular}}
\right]\!
=\!
\left[
\begin{tabular}{c}
$\delta_2$ \\
$\delta_3$ \\
$\delta_4$ \\
$\vdots$ \\
$\delta_{i_{\max}(n)-1}$ \\
$\delta_{i_{\max}(n)}$ \\
\end{tabular}
\right]
\end{equation}

Let $M$ be the $(i_{\max}(n)-1)\times (2\cdot i_{\max}(n)-3)$ matrix shown in (\ref{eq:syseq}). We express each entry $m_{i,j}$ in $M$ explicitly:
\be
m_{i,j}=
\begin{cases}
1 & \text{if }j=i\\
-2 & \text{if }j=i+1 \text{ and } j\leq i_{\max}(n)-1\\
1 & \text{if }j=i+3 \text{ and } j\leq i_{\max}(n)-1\\
1 & \text{if } j=i+i_{\max}(n)-2 \text{ and } j\geq 1\\
-1 &\text{if } j=i+i_{\max}(n)-1 \text{ and } j\leq 2\cdot i_{\max}(n)-3 \\
-1 &\text{if } j=i+i_{\max}(n) \text{ and } j\leq 2\cdot i_{\max}(n)-3\\
0 & \text{otherwise.}
\end{cases}
\ee
Let $A$ be the left $(i_{\max}(n)-1)\times (i_{\max}(n)-1)$ sub-matrix of $M$. Notice that $A$ is upper triangular and has 1's in the diagonal, so it is invertible.

To solve the equation in (\ref{eq:syseq}), we write $M$ in reduced row-echelon form:
\begin{equation}
\left[\small
\small\begin{tabular}{c c c c c c c | c c c c c c}
1 & 0 & 0 & 0 & $\cdots$ & 0 & 0 & 1 & 1 & 1 & $\cdots$ & 1 & 1 \\
0 & 1 & 0 & 0 & $\cdots$ & 0 & 0 & 1 & 1 & 1 & $\cdots$ & 1 & 1 \\
0 & 0 & 1 & 0 & $\cdots$ & 0 & 0 & 0 & 1 & 1 & $\cdots$ & 1 & 1 \\
$\vdots$&&&&$\cdots$ & & $\vdots$&$\vdots$&&& $\cdots$ & &$\vdots$\\
0 & 0 & 0 & 0 & $\cdots$ & 1 & 0 & 0 & 0 & 0 & $\cdots$ & 1 & 1 \\
0 & 0 & 0 & 0 & $\cdots$ & 0 & 1 & 0 & 0 & 0 & $\cdots$ & 0 & 1 \\
\end{tabular}
\right].
\end{equation}

Thus the solutions for the system of equations are in the form:
\begin{equation}
\left[\small
\begin{tabular}{c}
$MC_1$ \\
$MS_2$ \\
$MS_3$ \\
$\vdots$ \\
$MS_{i_{\max}(n)-1}$ \\
\hline
$MC_2$ \\
$MC_3$ \\
$\vdots$ \\
$MC_{i_{\max}(n)-1}$
\end{tabular}
\right]
=
\left[\small
\begin{tabular}{c}
    $A^{-1}
    \begin{bmatrix}
    \delta_2 \\
    \delta_3 \\
    \delta_4 \\
    \vdots \\
    \delta_{i_{\max}(n)} \\
    \end{bmatrix}$\\
    \hline
0\\
0\\
$\vdots$\\
0
\end{tabular}
\right]
+
MC_2
\left[\small
\begin{tabular}{c}
-1 \\
-1 \\
0 \\
$\vdots$ \\
0 \\
\hline
1 \\
0 \\
$\vdots$ \\
0
\end{tabular}
\right]+MC_3
\left[\small
\begin{tabular}{c}
-1 \\
-1 \\
-1 \\
$\vdots$ \\
0 \\
\hline
0 \\
1 \\
$\vdots$ \\
0
\end{tabular}
\right]+\cdots+MC_{i_{\max}(n)-1}
\left[\small
\begin{tabular}{c}
-1 \\
-1 \\
-1 \\
$\vdots$ \\
-1 \\
\hline
0 \\
0 \\
$\vdots$ \\
1
\end{tabular}
\right].
\end{equation}
The length of game path is the sum of all $MC$ and $MS$ terms, thus can be written as
\begin{equation}
\scr
\begin{bmatrix}
    1 & 1 & 1 & \cdots & 1 & 1& 1 & \cdots & 1
\end{bmatrix}
\left[
\begin{tabular}{c}
$MC_1$ \\
$MS_2$ \\
$MS_3$ \\
$\vdots$ \\
$MS_{i_{\max}(n)-1}$ \\
$MC_2$ \\
$MC_3$ \\
$\vdots$ \\
$MC_{i_{\max}(n)-1}$
\end{tabular}
\right]
\end{equation}
which is equal to
\begin{equation}\label{eq:lenraw}
\begin{bmatrix}
    1 & 1 & 1 & \cdots & 1
\end{bmatrix}\
A^{-1}\
    \begin{bmatrix}
    \delta_2 \\
    \delta_3 \\
    \delta_4 \\
    \vdots \\
    \delta_{i_{\max}(n)}
\end{bmatrix}
\ -\ MC_2-2MC_3-\ \cdots\ -(i_{\max}(n)-2)MC_{i_{\max}(n)-1}.
\end{equation}

We calculate $A^{-1}$ with Gauss-Jordan elimination, and let $a_{i,j}$ be the $i,j$-th entry of $A^{-1}$:
\begin{equation}
A^{-1}\ =\ \scr
\begin{bmatrix}
    1 & 2 & 4 & \cdots & a_{1,i_{\max}(n)-2} & a_{1,i_{\max}(n)-1}\\
    0 & 1 & 2 & \cdots & a_{2,i_{\max}(n)-2} & a_{2,i_{\max}(n)-1}\\
    0 & 0 & 1 & \cdots & a_{3,i_{\max}(n)-2} & a_{3,i_{\max}(n)-1}\\
    \vdots& & & \cdots &           & \vdots \\
    0 & 0 & 0 & \cdots & 1 & 2\\
    0 & 0 & 0 & \cdots & 0 & 1\\
\end{bmatrix}
\end{equation}
where $a_{i,j}=2a_{i,j-1}-a_{i,j-3}$ for all $1\leq i \leq i_{\max}(n)-1,\ 4\leq j\leq i_{\max}(n)-1$. Also observe that $a_{i+1,j+1}=a_{i,j}$.

We claim that $a_{1,j}=F_{j+1}-1$, and prove this with induction. First, we see that $a_{1,1}=F_2-1=1,\ a_{1,2}=F_3-1=2,\ a_{1,3}=F_4-1-4$, so this claim holds for $j=1,2,3$. Then suppose this claim holds for all $j'<j$, consider $a_{1,j}$.
\begin{equation}
a_{1,j}=2a_{1,j-1}-a_{1,j-3}=2F_{j}-F_{j-2}-1=F_{j}+F_{j-1}+F_{j-2}-F_{j-2}-1=F_{j+1}-1.
\end{equation}
By induction, our claim holds for all $j$.

Then we calculate $\begin{bmatrix}
    1 & 1 & 1 & \cdots & 1
\end{bmatrix}
A^{-1}$:
\begin{align}
&    \begin{bmatrix}
    1 & 1 & 1 & \cdots & 1
\end{bmatrix}
A^{-1}\nonumber\\
=&
\begin{bmatrix}
    1 & 1 & 1 & \cdots & 1
\end{bmatrix}
\begin{bmatrix}
    F_2-1 & F_3-1 & \cdots & F_{i_{\max}(n)-1} & F_{i_{\max}(n)}\\
    0 & F_2-1 & \cdots & F_{i_{\max}(n)-2} & F_{i_{\max}(n)-1}\\
    \vdots& & \cdots &           & \vdots  \\
    0 & 0 & \cdots & 0 & F_2-1
\end{bmatrix} \nonumber\\
=&
\begin{bmatrix}
    1 & 3 & 7 & \cdots & \ds\sum_{i=2}^{i_{\max}(n)-1}(F_i-1) & \ds\sum_{i=2}^{i_{\max}(n)}(F_i-1)
\end{bmatrix}.
\end{align}

We add an extra 0 before the sequence of $1,\ 3,\ 7,\dots$ and express it explicitly as $a_1=0,\ a_j=\ds\sum_{i=2}^{j}(F_i-1)$ for $j\geq 2$. Then we find a formula for $j\geq 2$:
\begin{equation}
  a_j=\sum_{i=2}^{j}(F_i-1)=\sum_{i=1}^{j}F_i-j=(F_{j}-1)+(F_{j+1}-1)-j=F_{j+2}-j-2  .
\end{equation}
Since $a_1=0=F_3-1-2$ is consistent with this formula, we conclude that $a_j=F_{j+2}-j-2$ for all $j$.

Now the game length given in (\ref{eq:lenraw}) becomes
\begin{equation}\label{eq:tightBound}
    \sum_{j=1}^{i_{\max}(n)} a_j\ \delta_j-MC_2-2MC_3-\cdots-(i_{\max}(n)-2)MC_{i_{\max}(n)-1}.
\end{equation}

Since all $MC$ terms are non-negative, we ignore them for the upper bound. Thus the upper bound on game length is
\begin{equation}\label{eq:formula}
    \sum_{j=1}^{i_{\max}(n)} a_j\ \delta_j=
    a_1\ \delta_1 +
    a_2\ \delta_2 +
    a_3\ \delta_3 +
    \cdots +
    a_{i_{\max}(n)}\ \delta_{i_{\max}(n)}.
\end{equation}
This is the bound we claimed in Theorem \ref{thm:bound}.

\smallskip

We now analyze the order of this bound.

We show that $a_j \leq \frac{3+\sqrt{5}}{2}\ F_{j}-j-\frac{1+\sqrt{5}}{2}$ using Binet's formula
\begin{equation}
    F_j=\frac{1}{\sqrt{5}}\left(\left(\frac{1+\sqrt{5}}{2}\right)^{j+1}-\left(\frac{1-\sqrt{5}}{2}\right)^{j+1}\right)
\end{equation}
and the formula $a_j=F_{j+2}-j-2$. Thus,
\begin{align}
    a_j&\ =\
    \frac{F_{j+2}}{F_j}\ F_j-j-2\nonumber\\
    &\ =\ \frac{\frac{1}{\sqrt{5}}\left(\left(\frac{1+\sqrt{5}}{2}\right)^{j+3}-\left(\frac{1-\sqrt{5}}{2}\right)^{j+3}\right)} {\frac{1}{\sqrt{5}}\left(\left(\frac{1+\sqrt{5}}{2}\right)^{j+1}-\left(\frac{1-\sqrt{5}}{2}\right)^{j+1}\right)}\ F_j-j-2\nonumber\\
    &\ =\ \frac{3+\sqrt{5}}{2}\ F_j+
    \ \left(\frac{1-\sqrt{5}}{2}\right)^{j+1} -j-2.
\end{align}
Since $-1<\frac{1-\sqrt{5}}{2}<0$, we have $\left(\frac{1-\sqrt{5}}{2}\right)^{j+1} \leq \left(\frac{1-\sqrt{5}}{2}\right)^2\ =\ \frac{3-\sqrt{5}}{2}$, so
\begin{equation}
    a_j\ \leq\ \frac{3+\sqrt{5}}{2}\ F_j+\frac{3-\sqrt{5}}{2}-j-2\ =\ \frac{3+\sqrt{5}}{2}\ F_j-j-\frac{1+\sqrt{5}}{2}.
\end{equation}

Therefore we have
\begin{align}
\sum_{j=1}^{i_{\max}(n)} a_j\ \delta_j
\ \leq\ &\sum_{j=1}^{i_{\max}(n)} \left(\frac{\sqrt{5}+3}{2}\ F_j-j-\frac{1+\sqrt{5}}{2}\right)\delta_j\nonumber\\
    \ =\ & \frac{\sqrt{5}+3}{2}\sum_{j=1}^{i_{\max}(n)} F_j\ \delta_j - \sum_{j=1}^{i_{\max}(n)} j\ \delta_j - \frac{1+\sqrt{5}}{2}\sum_{j=1}^{i_{\max}(n)}\delta_j \nonumber\\
    \ =\ & \frac{\sqrt{5}+3}{2}\ n\ -\ IZ(n)\ -\ \frac{1+\sqrt{5}}{2}Z(n).
\end{align}

In conclusion, the game length is at most $\frac{\sqrt{5}+3}{2}\ n\ -\ IZ(n)\ -\ \frac{1+\sqrt{5}}{2}Z(n)$.
\end{proof}
In this proof, equation (\ref{eq:tightBound}) is the exact game length and (\ref{eq:formula}) is the bound we gave in Theorem \ref{thm:bound}. From these two equations, we observe that our upper bound is strict if and only if the game on $n$ can be played with only splitting and combine 1 moves.
\begin{rek}
The same method can be used to calculate game length even if the game does not start from all $1$'s
by replacing $\delta_i$'s with the difference in number of $F_i$'s between starting state and final state at each position. Also if we take the $MS_i$ terms as free variables instead of pivot variables, we can use this method to calculate the lower bound of the game.
\end{rek}

We move on to the proof of Theorem \ref{thm:splitGame}.
The reason we are interested in games that can be played with only splitting moves is to identify for which $n$ is our upper bound in Theorem  \ref{thm:bound} strict.
\begin{lem}\label{lem:splitgame1}
If $n\ =\ F_k-1\ (k\geq2)$, then we can play the game with only split and combine 1 moves (starting from any game state).
\end{lem}
\begin{proof}
To prove this, we first prove that any game state (except the final one) in the game on $n\ =\ F_k-1\ (k\geq2)$ has at least two $F_i$'s for some $i$.

Suppose for the sake of contradiction that there is a game state that has at most one of any $F_i$ and is not the final game state.
Since this game state is not the final state, there are moves that we can apply. Since this game state has at most one of any $F_i$, we cannot apply any splitting moves. Thus, there must be some combining moves available. We apply the combining move with the largest index, say $F_{i-1}\wedge F_{i}\rightarrow F_{i+1}$. Note that we cannot have $F_{i+1}$ in this game state, or we would have chosen to combine $F_i$ and $F_{i+1}$. Thus after this move, we still have at most one of any $F_i$. We repeat this process until we reach the final state, and each game state we visit has at most one of any $F_i$.

Now we consider the Zeckendorf Decomposition of $n\ =\ F_k-1$.
It has to be in the form of $F_1+F_3+F_5+\cdots$ or $F_2+F_4+F_6+\cdots$ because if we add $1$ to these decompositions, we can get a Fibonacci number. By assumption, we reach the final state with a combining move. Let $F_{j-1}\wedge F_{j}\rightarrow F_{j+1}$ be the last step we took. Since there is a $F_{j+1}$ in the final decomposition of $n$, there must also be a $F_{j-1}$. Thus we had two $F_{j-1}$'s before the last step we took. However, we just showed that each game state we visit has at most one of any $F_i$, which is a contradiction.

Now that we have proved that any game state (except the final one) in the game on $n\ =\ F_k-1\ (k\geq2)$ has at least two $F_i$'s for some $i$, we know that we can apply a splitting or combine $1$ move at any game state until the game terminates. Thus, the game can be played with only splitting and combine $1$ moves.
\end{proof}

\begin{lem}\label{lem:splitgame2}
If we can play the game with only splitting and combining $1$ moves, then $n$\ can only be in the form $F_{k}-1\ (k\geq2)$.
\end{lem}

\begin{proof}
We consider all possible $n$ that are not in the form of $F_{k}-1\ (k\geq2)$ and divide them into three cases based on their final decomposition. In each case, we prove that the game on $n$ cannot be played with only splitting moves.

\vspace{4pt}
\textit{Case 1}: The smallest term in the final decomposition of $n$ is at least $F_3$.

Suppose that the game on $n$ can be played with only splitting and combine $1$ moves.
Since the smallest term in $n$'s final decomposition is at least $F_3$, we have to generate at least one $F_3$ at some point of the game. Thus there must exist split $2$ moves in the game.
Let $t$ be the last split $2$ move in the game.

At step $t$, one $F_1$ is generated. Since there are no $F_1$ in the final decomposition, we have to do a combine $1$ move to decrease the number of $F_1$'s, producing a $F_2$ with this move. Since there is no $F_2$ in the final decomposition, we have to do a split $2$ move, which contradicts our assumption that step $t$ is the last split $2$ move.

\vspace{4pt}
\textit{Case 2}: The first two terms in the final decomposition of $n$ are $F_1$ and $F_4$.

Suppose that the game on $n$ can be played with only splitting and combine $1$ moves.
Since we have $F_4$ in the final decomposition, we have to generate a $F_4$ at some point of the game.
Thus there exist split $3$ moves in the game. To do the split move, we need a $F_3$, so there must exist a split $2$ move. Let step $v$ be the last split $2$ move in the game.

At step $v$, one $F_3$ and one $F_1$ is generated.
Since there is no $F_3$ in the final decomposition, there exists a split $3$ move after step $v$, and we denote it step $t$. Then in step $t$, there is one $F_1$ generated.

Since both step $v$ and $t$ generates $F_1$, there exists a game state after step $v$ that contains $2\ F_1$. Therefore, there exists a combine $1$ move after step $v$. With this move, one $F_2$ is generated. Since there is no $F_2$ in the final decomposition, we have to do a split $2$ move to get rid of this $F_2$ which contradicts our assumption that $v$ is the last split $2$ move.

\vspace{4pt}
\textit{Case 3}: The smallest term of $n$'s decomposition is $F_1$ or $F_2$, and the first $2$ terms in the decomposition are not $F_1$ and $F_4$.

For any such $n$, suppose the game can be played without any combining moves except combine $1$.

In the following proof, we define the gap between two terms as the difference between their indices.
Notice that for any $n$ not in the form of $F_k-1$ $(k\geq 2)$, its decomposition either has a smallest term of $F_3$ or larger (\textit{Case 1}), or contains two consecutive terms with a gap of at least $3$. We make the following claim.

\smallskip
\textbf{Claim 1.}
If a game is played with only split or combine $1$ moves, then no game state contains two consecutive terms with a gap larger than $3$.

\smallskip
\textbf{Proof of Claim 1.} Since the first step of the game is always $F_1\wedge F_1\rightarrow F_2$, which only generates a gap of $1$, the claim is true for the first step. Suppose the claim is true for $m$th step. For the ($m+1$)th step, if we split $2$'s, it generates a gap of at most $2$, and the gap between $F_{3}$ and other terms larger than $F_{3}$ shrinks.

If we combine $1$'s, we generate a gap of $1$, and the gap between $F_{2}$ and other terms larger than $F_2$ shrinks.

Every time we split $F_{i}\wedge F_{i} \rightarrow F_{i+1} \wedge F_{i-2}$ ($i>2$), we generate a gap of $3$, which is the gap between $F_{n+1}$ and $F_{i-2}$, but the gaps between $F_{i-2}$ and any terms smaller than $F_{n-2}$ shrink, and the gaps between $F_{i+1}$ with any terms larger than $F_{i+1}$ shrink. In other words, every time we split, we generate a gap of $3$ but also make other gaps smaller.

Therefore, for all the situations above, the claim is true for ($m+1$)th step, and by induction Claim $1$ is true.

\smallskip
We now use this claim to finish our proof.
Note that for any $n$ in Case $3$, there will be at least two terms with a gap of at least $3$ and no terms in between them. From Claim $1$, we also proved that when we only use combine $1$ and splitting moves, we can never generate a gap that is more than $3$. Therefore, if $n$'s decomposition has a gap of more than $3$, then it is out of our consideration. For other $n$, we can find an $i$ such that $F_{i-2}$ and $F_{i+1}$ are two terms in the final decomposition, and there is neither $F_{i-1}$ nor $F_{i}$ in the final decomposition.

Since $F_{i+1}$ is in the final decomposition, there must exist a step where $S_{i}$ is performed. This is because the first time $F_{i+1}$ is generated in the game must come from the $S_{i}$.
So, let step $t$ be the last $S_i$.

In order for the $S_i$ happen, there must be at least two $F_{i}$'s generated prior. Thus, there must exist a step of splitting $F_{i-1}$ moves before step $t$ (this is because the first time $F_{i}$ is generated in the game must come from the splitting $F_{i-1}$ move).
So, let step $v$ be the last $S_{i-1}$ in the game.

Since $F_{i-2}$ is in the final decomposition, $F_{i-3}$ is not in the final decomposition (because the final decomposition does not contain two consecutive Fibonacci numbers). Note that step $v$ has generated one $F_{i-3}$, so there must be a $S_{i-3}$ after step $v$ to get rid of this $F_{i-3}$ (here, the combine $1$ move is also considered as a splitting move), and we can call it step $r$.

Since step $v$ has generated one $F_i$ and $F_i$ is not in the final decomposition, there must be a $S_i$ after step $v$ in order to get rid of this $F_i$, and we can call it step $s$.

Since both step $r$ and step $s$ are after step $v$ and each of them has generated one $F_{i-2}$, there must exist $2F_{i-2}$ at some point after step $v$. As a result, there must be a $S_{i-2}$ after step $v$, and we can call it step $w$.

Since step $w$ has generated one $F_{i-1}$ and $F_{i-1}$ is not in the final decomposition, there must be a $S_{i-1}$ after step $w$ in order to get rid of this $F_{i-1}$. In other words, there is a $S_{i-1}$ after step $v$, which contradicts with our assumption that step $v$ is the last occurrence of $S_{i-1}$.

Therefore, Lemma \ref{lem:splitgame2} is proved.
\end{proof}

The proof for Theorem \ref{thm:splitGame} follows directly from Lemma \ref{lem:splitgame1} and \ref{lem:splitgame2}.



\section{Future Work}\label{sec:conclusion}
It is worth noting that while our upper bound is sharp for some $n$, there is still plenty of room for it to be tightened. It was alluded to previously in the proof of \ref{thm:bound} that this could be done by quantifying the number of each combining moves, but such work is beyond the scope of this paper. In recent work, Cusenza et. al. \cite{CDHKKMMTYZ} investigated winning strategies for alliances of players in a multi-person generalization of the Zeckendorf game; one can similarly investigate the number of moves of various strategies in these settings.

Additionally, there are many ways the Zeckendorf game can be generalized (see for example \cite{BEFMgen}). 
With that in mind, we ask the following questions in regards to how our work relates to other similar games.
\begin{itemize}
    \item Can our methods for analysis of game bounds be performed on generalized games?
    \item Can we generalize the strategies suggested in this paper to achieve the longest or shortest game length in generalized games?
\end{itemize}




\ \\


\begin{thebibliography}{WWWGJMMNP}

\bibitem[BEFM1]{BEFMgen}
\newblock P. Baird-Smith, A. Epstein, K. Flynt and S. J. Miller, \emph{The
Generalized Zeckendorf Game}, the Fibonacci Quarterly (Proceedings of the
18th Conference) \textbf{57} (2019), no. 55, 1--15.
\bburl{https://arxiv.org/pdf/1809.04883}.

\bibitem[BEFM2]{BEFM}
\newblock P. Baird-Smith, A. Epstein, K. Flynt and S. J. Miller, \emph{The
Zeckendorf Game}, Combinatorial and Additive Number Theory III,
CANT, New York, USA, 2017 and 2018, Springer Proceedings in Mathematics \&
Statistics \textbf{297} (2020), 25--38.
\bburl{https://arxiv.org/pdf/1809.04881}.

\bibitem[CDHKKMMTYZ]{CDHKKMMTYZ}
A. Cusenza, A. Dunkelberg, K. Huffman, D. Ke, D. Kleber, S. J. Miller, C. Mizgerd, V. Tiwari, J. Ye and X. Zheng, \emph{Winning Strategy for the Multiplayer and Multialliance Zeckendorf Games}, preprint. \bburl{https://arxiv.org/pdf/2009.03708}.




\bibitem[Ho]{Ho} V. E. Hoggatt,  \emph{Generalized Zeckendorf theorem}, Fibonacci Quarterly \textbf{10} (1972), no. 1 (special issue on representations), pages 89--93.

\bibitem[Ke]{Ke} T. J. Keller,  \emph{Generalizations of Zeckendorf's theorem}, Fibonacci Quarterly \textbf{10} (1972), no. 1 (special issue on representations), pages 95--102.


\bibitem[LLMMSXZ]{LLMMSXZ}
\newblock R. Li, X. Li, S. J. Miller, C. Mizgerd, C. Sun, D. Xia, Z. Zhou (2020) \emph{Deterministic Zeckendorf Games, Part I}, to appear in the
Fibonacci Quarterly. \bburl{https://arxiv.org/pdf/2006.16457}.

\bibitem[MN]{MN}
S. J. Miller and A. Newlon, \emph{The Fibonacci Quilt Game}, Fibonacci Quarterly \textbf{58} (2020), no. 2, 157--168. \bburl{https://arxiv.org/pdf/1909.01938}.

\bibitem[MW1]{MW1}
S. Miller, Y. Wang, \emph{From Fibonacci Numbers to Central Limit Type Theorems}, Journal of Combinatorial Theory, Series A \textbf{119} (2012), no. 7, 1398--1413.

\bibitem[MW2]{MW2}
S. Miller, Y. Wang, \emph{Gaussian Behavior in Generalized Zeckendorf Decompositions}, Combinatorial and Additive Number Theory,  CANT 2011 and 2012 (Melvyn B. Nathanson, editor), Springer Proceedings in Mathematics \& Statistics (2014), 159--173.


\bibitem[Ze]{Ze}
E. Zeckendorf, Repr\'{e}sentation des nombres naturels par une somme des nombres de Fibonacci ou de nombres de Lucas, Bulletin de la Soci\'{e}t\'{e} Royale des Sciences de Li\`{e}ge {\bf 41} (1972), pages 179--182.

\end{thebibliography}
\end{document}